\documentclass{article}
\usepackage{amsfonts}
\usepackage{amsmath}

\newtheorem{theorem}{Theorem}

\newtheorem{corollary}[theorem]{Corollary}

\newtheorem{definition}[theorem]{Definition}

\newtheorem{proposition}[theorem]{Proposition}

\newenvironment{proof}[1][Proof]{\noindent \textbf{#1.} }{\  \rule{0.5em}{0.5em}}
\input{tcilatex}
\begin{document}

\title{Focal Representation of $k$-slant Helices in $\mathbb{E}^{m+1}$}
\author{G\"{u}nay \"{O}zt\"{u}rk, Bet\"{u}l Bulca, Beng\"{u} Bayram \& Kadri
Arslan}
\maketitle

\begin{abstract}
\ A focal representation of a generic regular curve $\gamma $ in $\mathbb{E}%
^{m+1}$ consists of the centers of the osculating hyperplanes. A $k$-slant
helix $\gamma $ in $\mathbb{E}^{m+1}$ is a (generic) regular curve whose
unit normal vector $V_{k}$ makes a constant angle with a fixed direction $%
\overrightarrow{U}$ in $\mathbb{E}^{m+1}$. \ In the present paper we proved
that if $\gamma $ is a $k$-slant helix in $\mathbb{E}^{m+1},$ then the focal
representation $C_{\gamma }$ of $\gamma $ in $\mathbb{E}^{m+1}$ is a $%
(m-k+2) $-slant helix in $\mathbb{E}^{m+1}.$
\end{abstract}

\section{Introduction}

\footnote{%
2000 Mathematics Subject Classifications: 53A04, 53C42
\par
Key words and phrases: Frenet curve, Focal curve, Slant helix.
\par
{}} Curves with constant slope, or so-called general helices (inclined
curves), are well-known curves in the classical differential geometry of
space curves. They are defined by the property that the tangent makes a
constant angle with a fixed line (the axis of the general helix) (see, \cite%
{Al}, \cite{Ba}, \cite{CIKH} and \cite{G}). In \cite{Ha}, the definition is
more restrictive: the fixed direction makes constant angle with these all
the vectors of the Frenet frame. It is easy to check that the definition
only works in the odd dimensional case. Moreover, in the same reference, it
is proven that the definition is equivalent to the fact the ratios $\frac{%
\kappa _{2}}{\kappa _{1}},\frac{\kappa _{4}}{\kappa _{3}},...$ , $\kappa _{i%
\text{ }}$being the curvatures, are constant. Further, J. Monterde has
considered the Frenet curves in $\mathbb{E}^{m}$ which have constant
curvature ratios (i.e., $\frac{\kappa _{2}}{\kappa _{1}},\frac{\kappa _{3}}{%
\kappa _{2}},\frac{\kappa _{4}}{\kappa _{3}}...$ are constant) \cite{Mo}.
The Frenet curves with constant curvature ratios are called ccr-curves.
Obviously, ccr-curves are a subset of generalized helices in the sense of 
\cite{Ha}. It is well known that \textit{curves with constant curvatures (}$%
W $-curves\textit{) }are\textit{\ }well-known\textit{\ }ccr-curves \cite{KL}%
, \cite{OAH}.

Recently, Izumiya and Takeuchi have introduced the concept of slant helix in
Euclidean $3$-space $\mathbb{E}^{3}$ by satisfying that the normal lines
make a constant angle with a fixed direction \cite{IT}. Further in \cite{AT}
Ali and Turgut considered the generalization of the concept of slant helix
to Euclidean $n$-space $\mathbb{E}^{n}$, and gave some characterizations for
a non-degenerate slant helix. As a future work they remarked that it is
possible to define a slant helix of type-k as a curve whose unit normal
vector $V_{k}$ makes a constant angle with a fixed direction $%
\overrightarrow{U}$ \cite{GCH}.

For a smooth curve (a source of light) $\gamma $ in $\mathbb{E}^{m+1}$, the
caustic of $\gamma $ $($defined as the envelope of the normal lines of $%
\gamma )$ is a singular and stratified hypersurface. The focal curve of $%
\gamma $ , $C_{\gamma }$ , is defined as the singular stratum of dimension $%
1 $ of the caustic and it consists of the centres of the osculating
hyperspheres of $\gamma $ . Since the centre of any hypersphere tangent to $%
\gamma $ at a point lies on the normal plane to $\gamma $ at that point, the
focal curve of $\gamma $ may be parametrised using the Frenet frame $\left(
t,n_{1},n_{2},...,n_{m}\right) $ of $\gamma $ as follows: 
\begin{equation*}
C_{\gamma }(\theta )=(\gamma +c_{1}n_{1}+c_{2}n_{2}+...+c_{m}n_{m})(\theta ),
\end{equation*}%
where the coefficients $c_{1},...,c_{m}$ are smooth functions that are
called focal curvatures of $\gamma $ \cite{V3}.

This paper is organized as follows: Section $2$ gives some basic concepts of
the Frenet curves in $\mathbb{E}^{m+1}$. Section $3$ tells about the focal
representation of a generic curve given with a regular parametrization in $%
\mathbb{E}^{m+1}$. Further this section provides some basic properties of
focal curves in $\mathbb{E}^{m+1}$ and the structure of their curvatures. In
the final section we consider k-slant helices in $\mathbb{E}^{m+1}$. We
prove that if $\gamma $ is a $k$-slant helix in $\mathbb{E}^{m+1}$ then the
focal representation $C_{\gamma }$ of $\gamma $ is $(m-k+2)$-slant helix in $%
\mathbb{E}^{m+1}.$

\section{Basic Concepts}

Let $\gamma =\gamma (s):I\rightarrow \mathbb{E}^{m+1}$ be a regular curve in 
$\mathbb{E}^{m+1}$, $($i.e., $\left \Vert \gamma ^{\prime }(s)\right \Vert $
is nowhere zero$)$ where $I$ is interval in $\mathbb{R}$. Then $\gamma $ is
called a \textit{Frenet curve of \ osculating order }$d,$ $(2\leq d\leq m+1)$
if $\gamma ^{\text{ }\prime }(s),$ $\gamma ^{\text{ }\prime \prime }(s),$...,%
$\gamma ^{(d)}(s)$ are linearly independent and $\gamma ^{\text{ }\prime
}(s),$ $\gamma ^{\text{ }\prime \prime }(s),$...,$\gamma ^{(d+1)}(s)$ are no
longer linearly independent for all $s$ in $I$ \cite{V3}. In this case, $%
Im(\gamma )$ lies in an $d$-dimensional Euclidean subspace of $\mathbb{E}%
^{m+1}.$ To each Frenet curve of rank $d$ there can be associated
orthonormal $d$-frame $\left \{ t,n_{1},...,n_{d-1}\right \} $ along $\gamma
,$ the Frenet $r$-frame, and $d-1$ functions $\kappa _{1},\kappa
_{2},...,\kappa _{d-1}$:$I\longrightarrow \mathbb{R}$, the Frenet curvature,
such that%
\begin{equation}
\left[ 
\begin{array}{c}
t^{^{\prime }} \\ 
n_{1}^{^{\prime }} \\ 
n_{2}^{^{\prime }} \\ 
... \\ 
n_{d-1}^{^{\prime }}%
\end{array}%
\right] =v\left[ 
\begin{array}{ccccc}
0 & \kappa _{1} & 0 & ... & 0 \\ 
-\kappa _{1} & 0 & \kappa _{2} & ... & 0 \\ 
0 & -\kappa _{2} & 0 & ... & 0 \\ 
... &  &  &  & \kappa _{d-1} \\ 
0 & 0 & ... & -\kappa _{d-1} & 0%
\end{array}%
\right] \left[ 
\begin{array}{c}
t \\ 
n_{1} \\ 
n_{2} \\ 
... \\ 
n_{d-1}%
\end{array}%
\right]  \label{A1}
\end{equation}%
where, $v$ is the speed of $\gamma .$ In fact, to obtain $%
t,n_{1},...,n_{d-1} $ it is sufficient to apply the Gram-Schmidt
orthonormalization process to $\gamma ^{\prime }(s),$ $\gamma ^{\prime
\prime }(s),$...,$\gamma ^{(d)}(s)$. Moreover, the functions $\kappa
_{1},\kappa _{2},...,\kappa _{d-1}$ are easily obtained as by-product during
this calculation. More precisely, $t,n_{1},...,n_{d-1}$ and $\kappa
_{1},\kappa _{2},...,\kappa _{d-1}$ are determined by the following formulas:%
\begin{eqnarray}
v_{1}(s) &:&=\gamma ^{\text{ }\prime }(s)\  \  \text{\  \ };t:=\frac{v_{1}(s)}{%
\left \Vert v_{1}(s)\right \Vert },  \notag \\
v_{\alpha }(s) &:&=\gamma ^{(\alpha )}(s)-\sum_{i=1}^{\alpha -1}<\gamma
^{(\alpha )}(s),v_{i}(s)>\frac{v_{i}(s)}{\left \Vert v_{i}(s)\right \Vert
^{2}},  \label{A2} \\
\kappa _{\alpha -1}(s) &:&=\frac{\left \Vert v_{\alpha }(s)\right \Vert }{%
\left \Vert v_{\alpha -1}(s)\right \Vert \left \Vert v_{1}(s)\right \Vert },
\notag \\
n_{\alpha -1} &:&=\frac{v_{\alpha }(s)}{\left \Vert v_{\alpha }(s)\right
\Vert },  \notag
\end{eqnarray}%
where $\alpha \in \left \{ 2,3,...,d\right \} $ (see, \cite{G})$.$

A Frenet curve of rank $d$ for which $\kappa _{1},\kappa _{2},...,\kappa
_{d-1}$ are constant is called (generalized) screw line or helix \cite{CDV}.
Since these curves are trajectories of the $1$-parameter group of the
Euclidean transformations, so, F. Klein and S. Lie called them $W$\textit{%
-curves} \cite{KL}. For more details see also \cite{Ch}. $\gamma $ is said
to have constant curvature ratios (that is to say, it is a ccr-curve) if all
the quotients $\frac{\kappa _{2}}{\kappa _{1}},\frac{\kappa _{3}}{\kappa _{2}%
},\frac{\kappa _{4}}{\kappa _{3}},...,\frac{\kappa _{i}}{\kappa _{i-1}}$ $%
(1\leq i\leq m-1$) are constant \cite{Mo}, \cite{OAH}.

\section{The Focal Representation of a Curve in $\mathbb{E}^{m+1}$}

The hyperplane normal to $\gamma $ at a point is the union of all lines
normal to $\gamma $ at that point. The envelope of all hyperplanes normal to 
$\gamma $ is thus a component of the focal set that we call the main
component (the other component is the curve $\gamma $ itself, but we will
not consider it) \cite{V1}.

\begin{definition}
Given a generic curve (i.e., a Frenet curve of osculating orde $m+1$) $%
\gamma :\mathbb{R}\rightarrow \mathbb{E}^{m+1},$ let $F:\mathbb{E}%
^{m+1}\times \mathbb{R}\rightarrow \mathbb{R}$ be the $(m+1)$-parameter
family of real functions given by%
\begin{equation}
F(q,\theta )=\frac{1}{2}\left \Vert q-\gamma (\theta )\right \Vert ^{2}.
\label{a2}
\end{equation}%
The caustic of the family $F$ is given by the set%
\begin{equation}
\left \{ q\in \mathbb{E}^{m+1}:\exists \theta \in \mathbb{R}:F_{q}^{\prime
}(\theta )=0\text{ and }F_{q}^{\prime \prime }(\theta )=0\right \} .
\label{a3}
\end{equation}%
\cite{V1}.
\end{definition}

\begin{proposition}
\cite{V2} The caustic of the family $F(q,\theta )=\frac{1}{2}\left \Vert
q-\gamma (\theta )\right \Vert ^{2}$ coincides with the focal set of the
curve $\gamma :\mathbb{R}\rightarrow \mathbb{E}^{m+1}.$
\end{proposition}

\begin{definition}
The center of the osculating hypersphere of $\gamma $ at a point lies in the
hyperplane normal to the $\gamma $ at that point. So we can write%
\begin{equation}
C_{\gamma }=\gamma +c_{1}n_{1}+c_{2}n_{2}+\cdots +c_{m}n_{m},  \label{a4}
\end{equation}%
which is called focal curve of $\gamma $, where $c_{1},c_{2},\ldots ,c_{m}$
are smooth functions of the parameter of the curve $\gamma $. We call the
function $c_{i}$ the $i^{th}$ focal curvature of $\gamma .$ Moreover, the
function $c_{1}$ never vanishes and $c_{1}=\frac{1}{\kappa _{1}}$ \cite{V3}.
\end{definition}

The focal curvatures of $\gamma $, parametrized by arc length s, satisfy the
following "scalar Frenet equations" for $c_{{\small m}}\neq 0:$%
\begin{eqnarray}
1 &=&\kappa _{1}c_{1}  \notag \\
c_{1} &=&\kappa _{2}c_{2}  \notag \\
c_{2} &=&-\kappa _{2}c_{{\small 1}}+\kappa _{3}c_{3}  \notag \\
&&\text{...}  \label{a5} \\
c_{m-1} &=&-\kappa _{m-1}c_{m-2}+\kappa _{m}c_{m}  \notag \\
c_{m}-\frac{({R}_{m}^{2}\acute{)}}{2c_{m}} &=&-\kappa _{m}c_{m-1}  \notag
\end{eqnarray}%
where $R_{{\small m}}$\ the radius of the osculating m-sphere. In particular 
$R_{m}^{2}=\left \Vert \QTR{sl}{C}_{\gamma }{-}\gamma \right \Vert ^{2}$ 
\cite{V3}.

\begin{theorem}
\cite{V1} Let $\gamma :s\rightarrow \gamma (s)\in \mathbb{E}^{m+1}$ be a
regular generic curve . Write $\kappa _{1},\kappa _{2},...,\kappa _{m}$ for
its Euclidean curvatures and $\left \{ t,n_{1},n_{2},...,n_{m}\right \} $
for its Frenet Frame. For each $\gamma (s)$ of $\gamma ,$ write $\varepsilon
(s)$ for the sign of $(c_{m}^{\prime }+c_{m-1}\kappa _{m})(s)$ and $\delta
_{\alpha }(s)$ for the sign of $(-1)^{\alpha }\varepsilon (s)\kappa _{m}(s),$
$\alpha =1,\ldots ,m.$ Then the following holds:

a) The Frenet frame $\left \{ T,N_{1},N_{2},...,N_{m}\right \} $ of $%
C_{\gamma }$ at $C_{\gamma }(s)$ is well-defined and its vectors are given
by $T=\varepsilon n_{m},$ $N_{\alpha }=\delta _{\alpha }n_{m-l},$ for $%
l=1,\ldots ,m-1,$ and $N_{m}=\pm t.$ The sign in $\pm t$ is chosen in order
to obtain a positive basis.

b) The Euclidean curvatures $K_{1},K_{2},...,K_{m}$ of the parametrized
focal curve of $\gamma .$ $C_{\gamma }:s\rightarrow C_{\gamma }(s)$, are
related to those of $\gamma $ by:%
\begin{equation}
\frac{K_{1}}{\left \vert \kappa _{m}\right \vert }=\frac{K_{2}}{\kappa _{m-1}%
}=\cdots =\frac{\left \vert K_{m}\right \vert }{\kappa _{1}}=\frac{1}{\left
\vert c_{m}^{\prime }+c_{m-1}\kappa _{m}\right \vert },  \label{a6}
\end{equation}%
the sign of $K_{m}$ is equal to $\delta _{m}$ times the sign chosen in $\pm
t.$
\end{theorem}

That is the Frenet formulas of $C_{\gamma }$ at $C_{\gamma }(s)$ is%
\begin{eqnarray}
T^{^{\prime }} &=&\frac{1}{A}\left \vert {\small \kappa }_{m}\right \vert
N_{1}  \notag \\
N_{1}^{^{\prime }} &=&\frac{1}{A}\left( {\small -}\left \vert {\small \kappa 
}_{m}\right \vert T+{\small \kappa }_{m-1}N_{2}\right)  \notag \\
N_{2}^{^{\prime }} &=&\frac{1}{A}\left( {\small -}\left \vert {\small \kappa 
}_{m-1}\right \vert N_{1}+{\small \kappa }_{m-2}N_{3}\right)  \label{a7} \\
&&...  \notag \\
N_{m-1}^{^{\prime }} &=&\frac{1}{A}\left( {\small -\kappa }_{2}N_{m-2}%
{\small \mp \delta }_{m}{\small \kappa }_{1}N_{m}\right)  \notag \\
N_{m}^{^{\prime }} &=&\frac{1}{A}{\small \mp \delta }_{m}{\small \kappa }%
_{1}N_{m-1}  \notag
\end{eqnarray}

where $A=\left \vert c_{m}^{\prime }+c_{m-1}\kappa _{m}\right \vert .$

\begin{corollary}
Let $\gamma =\gamma (s)$ be a regular generic curve in $\mathbb{E}^{m+1}$
and $C_{\gamma }:s\rightarrow C_{\gamma }(s)$ be focal representation of $%
\gamma .$ Then the Frenet frame of $C_{\gamma }$ becomes as follows;

i) If $m$ is even 
\begin{eqnarray}
T &=&n_{m}  \notag \\
N_{1} &=&-n_{m-1}  \notag \\
N_{2} &=&n_{m-2}  \notag \\
&&...  \label{a8} \\
N_{m-1} &=&-n_{1}  \notag \\
N_{m} &=&t  \notag
\end{eqnarray}

ii) If $m$ is odd,%
\begin{eqnarray}
T &=&n_{m}  \notag \\
N_{1} &=&-n_{m-1}  \notag \\
N_{2} &=&n_{m-2}  \notag \\
&&...  \label{a9} \\
N_{m-1} &=&n_{1}  \notag \\
N_{m} &=&-t.  \notag
\end{eqnarray}
\end{corollary}

\begin{proof}
By the use of (\ref{a6}) with (\ref{a7}) we get the result.
\end{proof}

\section{k-Slant Helices}

Let $\gamma =\gamma (s):I\rightarrow $ $\mathbb{E}^{m+1}$ be a regular
generic curve given with arclenth parameter. Further, let $\overrightarrow{U}
$ be a unit vector field in $\mathbb{E}^{m+1}$ such that for each $s\in I$ \
the vector $\overrightarrow{U}$ is expressed as the linear combinations of
the orthogonal basis $\left \{ V_{1}(s),V_{2}(s),...,V_{m+1}(s)\right \} $
with%
\begin{equation}
\overrightarrow{U}=\underset{j=1}{\overset{m+1}{\dsum }}a_{j}(s)V_{j}(s).
\label{D1}
\end{equation}%
where $a_{j}(s)$ are differentiable functions $a_{k}(s)$, $1\leq j\leq m+1.$

Differentiating $\overrightarrow{U}$ and using the Frenet equations (\ref{A1}%
) one can get

\begin{equation}
\frac{d\overrightarrow{U}}{ds}=\underset{i=1}{\overset{m+1}{\dsum }}%
P_{i}(s)V_{i}(s),  \label{D2}
\end{equation}%
where 
\begin{eqnarray}
P_{1}(s) &=&a_{1}^{^{\prime }}-\kappa _{1}a_{2},  \label{D3} \\
P_{i}(s) &=&a_{i}^{^{\prime }}+\kappa _{i-1}a_{i-1}-\kappa _{i}a_{i+1},\text{
}2\leq i\leq m,  \notag \\
P_{m+1}(s) &=&a_{m+1}^{^{\prime }}+\kappa _{m}a_{m}  \notag
\end{eqnarray}

If the vector field $\overrightarrow{U}$ is constant then the following
system of ordinary differential equations are obtained 
\begin{eqnarray}
0 &=&a_{1}^{^{\prime }}-\kappa _{1}a_{2},  \notag \\
0 &=&a_{2}^{^{\prime }}+\kappa _{1}a_{1}-\kappa _{2}a_{3},  \label{C4} \\
0 &=&a_{i}^{^{\prime }}+\kappa _{i-1}a_{i-1}-\kappa _{i}a_{i+1},\text{ }%
3\leq i\leq m,  \notag \\
0 &=&a_{m+1}^{^{\prime }}+\kappa _{m}a_{m}.  \notag
\end{eqnarray}

\begin{definition}
Recall that a unit speed generic curve $\gamma =\gamma (s):I\rightarrow $ $%
\mathbb{E}^{m+1}$ is called a $k$-type slant helix if the vector field $%
V_{k} $ ($1\leq k\leq m+1$) makes a constant angle $\theta _{k}$ with the
fixed direction $\overrightarrow{U}$ in $\mathbb{E}^{m+1}$, that is 
\begin{equation}
<\overrightarrow{U},V_{k}>=\cos \theta _{k},\text{ }\theta _{k}\neq \frac{%
\pi }{2}\text{. }  \label{C5}
\end{equation}
\end{definition}

1-type slant helix is known as generalized, or cylindrical helix \cite{AL}
or generalized helix \cite{MC}, \cite{Ba}. For the characterization of
generalized helices in $(n+2)$-dimensional Lorentzisan space $\mathbb{L}%
^{n+2}$ see \cite{YH}.

We give the following result;

\begin{theorem}
Let $\gamma =\gamma (s)$ be a regular generic curve in $\mathbb{E}^{m+1}.$
If $C_{\gamma }:s\rightarrow C_{\gamma }(s)$ is a focal representation of $%
\gamma $ then the following statements are valid;

i) If $\gamma $ is $1$-slant helix then the focal representation $C_{\gamma
} $ of $\gamma $ is $(m+1)$-slant helix in $\mathbb{E}^{m+1}.$

ii) If $\gamma $ is $(m+1)$-slant helix then the focal representation $%
C_{\gamma }$ of $\gamma $ is $1$-slant helix in $\mathbb{E}^{m+1}.$

iii) If $\gamma $ is k-slant helix $(2<k<m)$ then the focal representation $%
C_{\gamma }$ of $\gamma $ is $(m-k+2)$-slant helix in $\mathbb{E}^{m+1}.$
\end{theorem}

\begin{proof}
i) Suppose $\gamma $ is a $1$-slant helix in $\mathbb{E}^{m+1}.$ Then by
Definition $6$ the vector field $V_{1}$ makes a constant angle $\theta _{1}$
with the fixed direction $\overrightarrow{U}$ defined in (\ref{D1}), that is 
\begin{equation}
<\overrightarrow{U},V_{1}>=\cos \theta _{1},\text{ }\theta _{1}\neq \frac{%
\pi }{2}\text{. }  \label{C6}
\end{equation}%
For a focal representation $C_{\gamma }(s)$ of $\gamma $ we can chose the
orthogonal basis 
\begin{equation*}
\left \{ V_{1}(s)=t,V_{2}(s)=n_{1},...,V_{m+1}(s)=n_{m}\right \}
\end{equation*}%
such that the equalities (\ref{a8}) or (\ref{a9}) is hold$.$ Hence, we get,%
\begin{equation}
<\overrightarrow{U},V_{1}>=<\overrightarrow{U},t>=<\overrightarrow{U},\pm
N_{m}>=cons.  \label{C7}
\end{equation}%
where $\left \{ T,N_{1},N_{2},...,N_{m}\right \} $ is the Frenet frame of $%
C_{\gamma }$ at point $C_{\gamma }(s).$ From the equality (\ref{C7}) it is
easy to say that $C_{\gamma }$ is a (m+1)-slant helix of $\mathbb{E}^{m+1}.$

ii) Suppose $\gamma $ is a $(m+1)$-slant helix in $\mathbb{E}^{m+1}.$ Then
by Definition $6$ the vector field $V_{m+1}$ makes a constant angle $\theta
_{m+1}$ with the fixed direction $\overrightarrow{U}$ defined in (\ref{D1}),
that is 
\begin{equation}
<\overrightarrow{U},V_{m+1}>=\cos \theta _{m+1},\text{ }\theta _{m+1}\neq 
\frac{\pi }{2}\text{. }  \label{C8}
\end{equation}%
For a focal representation $C_{\gamma }(s)$ of $\gamma $ one can get 
\begin{equation}
<\overrightarrow{U},V_{m+1}>=<\overrightarrow{U},n_{m}>=<\overrightarrow{U}%
,T>=cons.  \label{C9}
\end{equation}%
where $\left \{ V_{1}=t,V_{2}=n_{1},...,V_{m+1}=n_{m}\right \} $ and $%
\left
\{ T,N_{1},N_{2},...,N_{m}\right \} $ are the Frenet frame of $\gamma 
$ and $C_{\gamma }$ respectively$.$ From the equality (\ref{C9}) it is easy
to say that $C_{\gamma }$ is a 1-slant helix of $\mathbb{E}^{m+1}.$

iii) Suppose $\gamma $ is a $k$-slant helix in $\mathbb{E}^{m+1}$ $(2\leq
k\leq m).$ Then by Definition $6$ the vector field $V_{k}$ makes a constant
angle $\theta _{k}$ with the fixed direction $\overrightarrow{U}$ defined in
(\ref{D1}), that is 
\begin{equation}
<\overrightarrow{U},V_{k}>=\cos \theta _{k},\text{ }\theta _{k}\neq \frac{%
\pi }{2}\text{, }2\leq k\leq m\text{.}  \label{C10}
\end{equation}%
Let $C_{\gamma }(s)$ be a focal representation of of $\gamma $. Then using
the equalities (\ref{a8}) or (\ref{a9}) we get 
\begin{equation}
<\overrightarrow{U},V_{k}>=<\overrightarrow{U},n_{k-1}>=<\overrightarrow{U}%
,N_{m-k+1}>=cons.,\text{ }2\leq k\leq m  \label{C11}
\end{equation}%
where 
\begin{equation*}
\left \{ V_{1}=t,V_{2}=n_{1},...,V_{m+1}=n_{m}\right \}
\end{equation*}%
and 
\begin{equation*}
\left \{ \widetilde{V}_{1}=T,\widetilde{V}_{2}=N_{1},...,\widetilde{V}%
_{m-k+2}=N_{m-k+1},...,\widetilde{V}_{m+1}=N_{m}\right \}
\end{equation*}%
are the Frenet frame of $\gamma $ and $C_{\gamma }$ respectively$.$ From the
equality (\ref{C11}) it is easy to say that $C_{\gamma }$ is a ($m-k+2$%
)-slant helix of $\mathbb{E}^{m+1}.$
\end{proof}

\begin{tabular}{l}
G\"{u}nay \"{O}zt\"{u}rk \\ 
Department of Mathematics \\ 
Kocaeli University \\ 
41380, Kocaeli, TURKEY \\ 
e-mail: ogunay@kocaeli.edu.tr%
\end{tabular}

\begin{tabular}{l}
Kadri Arslan \& Bet\"{u}l Bulca \\ 
Department of Mathematics \\ 
Uluda\u{g} University \\ 
16059 Bursa, TURKEY \\ 
e-mail: arslan@uludag.edu.tr \\ 
e-mail: bbulca@uludag.edu.tr%
\end{tabular}

\begin{tabular}{l}
Beng\"{u} Bayram \\ 
Department of Mathematics \\ 
Bal\i kesir University \\ 
Bal\i kesir, TURKEY \\ 
e-mail: benguk@bal\i kesir.edu.tr%
\end{tabular}

\end{document}